\documentclass[preprint,12pt]{elsarticle}

\usepackage{amssymb}
\usepackage{amsmath}
\usepackage{amsthm}

\newtheorem{theorem}{Theorem}[section]
\newtheorem{lemma}[theorem]{Lemma}
\newtheorem{definition}[theorem]{Definition}
\newtheorem{example}{Example}
\newtheorem{remark}[theorem]{Remark}

\begin{document}

\begin{frontmatter}



\title{Maximum principle for optimal control of interacting particle system: stochastic flow  model}

\author[label1]{Andrey A. Dorogovtsev}
\affiliation[label1]{organization={Institute of Mathematics, NAS of Ukraine,  Kyiv 01024, Ukraine}}

\author[label2]{Yuecai Han}
\affiliation[label2]{organization={School of Mathematics, Jilin University, Changchun 130012, China}}

\author[label2,label1]{Kateryna Hlyniana}

\author{Yuhang Li\corref{cor1}\fnref{label2}}
\ead{yuhangl22@mails.jlu.edu.cn}
\cortext[cor1]{Corresponding: }

\begin{abstract}
In this paper, we study a stochastic optimal control problem for an interacting particle system. By introducing a generalized backward stochastic differential equation with interaction, we establish a stochastic maximum principle for the optimal control. We prove the existence and uniqueness of the solution to this class of equations and derive necessary conditions that an optimal control must satisfy. As an application, we examine the linear-quadratic case to illustrate the main results.
\end{abstract}




\end{frontmatter}



\section{Introduction}
In this paper, we investigate a stochastic optimal control problem for a family of particles whose dynamics are governed by a stochastic differential equation with interaction. Such equations were introduced by A.A. Dorogovtsev \cite{Dor2003, dorogovtsev2023measure}. We consider the following form of the stochastic differential equation (SDE) with interactions:
\begin{equation}
\label{SDE_WI}
    \left\{\begin{array}{l}
 d X(u, t)=b\left(t,X(u, t), \mu_{t}\right) d t+\int_{\mathbb{R}^d}\sigma(t, X(u,t), \mu_t,r)W(dt,dr), \\
 X(u, 0)=u,\quad u \in \mathbb{R}^{d}, \\
 \mu_{t}=\mu_{0} \circ X(\cdot, t)^{-1}.
\end{array}\right.
\end{equation}
Here,  $W(\cdot,\cdot)$ is an $m$-dimensional standard Wiener sheet  and the diffusion coefficient $\sigma$  takes values in $\mathbb{R}^{d\times m}$.
Equation  (\ref{SDE_WI}) describes the evolution of a family of particles $\{X(u, \cdot)\}_{ u \in \mathbb{R}^{d}}$ each starting from a point of the space $u\in\mathbb{R}^{d}$. The initial mass distribution of the family of particles is given by a probability measure $\mu_{0}$ which is supposed to be a probability measure on the Borel $\sigma-$algebra of $\mathbb{R}^{d}$. The mass distribution of the family of particles evolves over time, and at moment $t$, it can be written as the push forward of $\mu_{0}$ by the mapping $X(\cdot, t)$. Both the drift and diffusion coefficients in Equation (\ref{SDE_WI}) depend on the mass distribution $\mu_{t}$. This gives the possibility to describe the motion of a particle that depends on the positions of all other particles in the system. Note that the measure-valued process $\mu_t$ is random,  and we may also write $\mu_t(\Delta)=\mu_t(\omega,\Delta)$.

Analogously to the classical stochastic optimal control framework, we consider the following control problem for an SDE with interaction:
\begin{equation}
\left\{\begin{array}{l}
\label{Control_problem_with_int}
d X(u, t)=b\left(t, X(u, t), \mu_{t}, \alpha_{t}\right) d t+ \int_{\mathbb R^d}\sigma\left(t, X(u, t), \mu_{t}, \alpha_{t}, r\right) W(dt,dr), \\
X(u, 0)=u, \\
\mu_{t}=\mu_{0} \circ X(\cdot, t)^{-1} .
\end{array}\right.
\end{equation}
The evolution of the system (\ref{Control_problem_with_int}) depends on a stochastic control process $\mathbf{\alpha}=\left(\alpha_{t},\quad  0 \leq t \leq T\right)$, which is assumed to be progressively measurable with respect to filtration  $\{\mathcal F_{t}\}_{t\geq 0}$ generated by the Wiener sheet.
At each time $t,$  the control value $\alpha_{t}$ is chosen to influence the system's dynamics.
The goal of the stochastic optimal control problem is to find a control process $\left(\alpha_t\right)_{0\leq t\leq T}$ that minimizes the cost functional $J(\alpha)$ of the form
\begin{equation}
\label{cost_function}
J(\alpha)=\mathbb{E} \left [ \int_{0}^{T} f\left(t, \mu_{t}, \alpha_{t}\right) d t+g\left( \mu_{T}\right)\right]  .
\end{equation}

The investigation of this optimal control problem is motivated by the Monge-Kantorovich transport problem \cite{rachev1985monge,bogachev2012monge,leonard2012schrodinger}. The Monge-Kantorovich transport problem deals with the optimal mapping of one probability measure on $\mathbb R^d$ to another while minimizing a transportation cost functional. In our setting, the system evolves from an initial mass distribution $\mu_0$. The cost functional (\ref{cost_function}) of the control problem consists of two components:  a running cost function, given by the integral of $f$ over time, and a terminal cost $g$, which depends on the mass distribution of the system at the final time.
In analogy with the Monge–Kantorovich problem, the function $g$ may be interpreted as a measure of distance between a target probability measure $\nu$  and the terminal distribution $\mu_T$ of the system. For instance, we may take $g(\mu_T) = \rho(\mu_T, \nu),$ 
where $\rho$ denotes a distance between measures. One possible choice is:
\begin{align}\label{dis}
\rho^2(\mu,\nu)=\int\int_{\mathbb{R}^2}\Gamma(u-v)[\mu(du)-\nu(du)]\cdot[\nu(dv)-\mu(dv)],
\end{align}
where $\Gamma$ is a continuous, non-negative definite function. Unlike the classical Monge-Kantorovich transport problem, however, the terminal mass distribution $\mu_T=\mu_T(\omega)$ is random in our setting.  Consequently, the mimimization of $\rho(\mu_T,\nu)$
 must be understood in the sense of expected value.

In the case where the initial measure is discrete, say
$$\mu_{0}=\sum_{i=1}^n p_i\delta_{u_i}, \ u_i\in \mathbb{R}, 1\le i \le n,$$
the mass distribution of the system at time $t$ is equal to 
$$ \mu_{t}=\sum_{i=1}^np_i\delta_{X(u_i, t)}.$$
From this one can see that the dependence of coefficients on the measure $\mu_t$ in equation (\ref{Control_problem_with_int}) and in the cost function (\ref{cost_function}) is realized through the positions of the of the `heavy points' $X(u_1,t),\ldots, X(u_n,t).$ The control problem for the stochastic differential equation with interaction, in this case, can be solved by considering the system of equations for heavy points. To write it, let us denote by 
$$
b_i(t, u_1,\ldots, u_n,\alpha):= b\left(t, u_i, \sum_{j=1}^n p_j \delta_{u_j},\alpha\right),
$$
$$
\sigma_i(t, u_1,\ldots, u_n,\alpha,r):=\sigma \left(t, u_i, \sum_{j=1}^n p_j \delta_{u_j},\alpha,r\right),
$$
Then, the corresponding control problem for heavy points takes the form:
\begin{equation}
\label{n_pont_control}
    \left\{\begin{array}{l}
 d X(u_i, t)=b_i\left(t,X(u_1, t),\ldots, X(u_n,t),\alpha_t\right) d t\\\qquad\qquad\quad+\int_{\mathbb{R}^d}\sigma_i(t, X(u_1,t), \ldots, X(u_n,t),\alpha_t,r)W(dt,dr), \quad i=1,\ldots, n, \\
 X(u_i, 0)=u_i,\quad i=1,\ldots,n.\,
\end{array}\right.
\end{equation}
The associated cost functional becomes:
\begin{equation}
\label{n_point_cost_function}
J(\alpha)=\mathbb{E}\left [ \int_{0}^{T} f\left( \sum_{i=1}^np_i\delta_{X(u_i, t)}, \alpha_{t}\right) d t+g\left( \sum_{i=1}^np_i\delta_{X(u_i, T)}\right)\right] .
\end{equation}
Once an optimal control process  $(\alpha_t)_{0\leq t\leq T}$ is found, the SDE with interaction  (\ref{Control_problem_with_int}) can be solved for any remaining points $R^d\setminus\{u_1, \ldots, u_n\}$ with zero mass, i.e. for  $X(u,\cdot).$ 

It is important to emphasize that the stochastic differential equations (SDEs) with interaction and backward stochastic differential equations (BSDEs) with interaction considered in this work differ significantly from the classical McKean–Vlasov equations, which have been extensively studied by many researchers\cite{Kotelenez95, Buckdahn09, AGRAM2022, kolokoltsov_2010}. 
The key distinction lies in the nature of the measure dependence: in McKean–Vlasov equations, the coefficients depend on the distribution law of the position of a representative particle, whereas in SDEs with interaction, they depend on the distribution of mass (wich is given by a random probability measure) of the system.
More precisely, these two frameworks describe particle systems from fundamentally different perspectives. In the  McKean-Vlasov equation, each particle in the system is driven by independent noises, and their initial positions are i.i.d. variables. The empirical distribution of the position for these particles converges to the probability distribution law $\mathcal{L}_t$. By contrast, in SDEs with interaction, the initial positions of particles are determined and the distribution of their initial position is known. Moreover, the particles are not driven by independent noise sources but by a common Brownian sheet. This results in spatially dependent stochastic perturbations: particles located at different positions experience different random influences.  To illustrate the distinction, consider a system of  $N$ particles.The McKean–Vlasov-type system is described by the following SDEs:
\begin{align}\label{mv}
\left\{\begin{array}{ll}
dX^i(t)=b\left(t,X^i(t),\frac{1}{N}\sum_{j=1}^N\delta_{X^j(t)}\right)dt+\sigma\left(t,X^i(t),\frac{1}{N}\sum_{j=1}^N\delta_{X^j(t)}\right)dW^i_t,
\\X^i(0)=x^i_0,
\end{array}\right.
\end{align}
for $1\le i\le N$, where $x^i_0$ are i.i.d variables with the probability distribution 
$\mathcal{L}_0$ and $W_t^i$ are independent Brownian motions.
 In contrast, for interaction type, taking the initial measure $\mu_0=\frac{1}{N}\sum_{i=1}^N\delta_{u^i}$ with determined points $u^i,$  the SDE with interaction
can be written as a system of stochastic differential equations for heavy particles $X(u^i,t):$
\begin{align}\label{int-discrete}
\left\{\begin{array}{ll}
dX(u^i,t)=b\left(t,X(u^i,t),\frac{1}{N}\sum_{j=1}^N\delta_{X(u^j,t)}\right)dt\\
\qquad\qquad\quad+\int_{\mathbb{R}^d}\sigma\left(t,X(u^i,t),\frac{1}{N}\sum_{j=1}^N\delta_{X(u^j,t)},q\right)W(dq,dt),
\\X(u^i,0)=u^i, \ i= 1,\ldots, N.
\end{array}\right.
\end{align}
Letting $N$ tends to $+\infty$, the limiting dynamics in $L^2_{\mathcal{F}}([0,T],\mathbb{R}^d)$ for the McKean–Vlasov system \eqref{mv} is:
\begin{align*}
\left\{\begin{array}{ll}
dX(t)=b\left(t,X(t),\mathcal{L}_t\right)dt+\sigma\left(t,X(t),\mathcal{L}_t\right)dW_t,
\\X(0)=x_0,\quad x_0\sim \mathcal{L}_0,\\
\mathcal{L}_t=\mathbb P \circ X(t)^{-1},
\end{array}\right.
\end{align*}
where $x_0$ is independent with $W_t$ and $\mathbb P \circ X(t)^{-1}$ denotes probability distribution of $X(t). $ In contrast, the limiting form of the interacting system \eqref{int-discrete} is:
\begin{equation*} 
\left\{\begin{array}{l}
d X(u, t)=b\left(t, X(u, t), \mu_t\right) dt\\
\qquad\qquad\quad+\int_{\mathbb{R}^d}\sigma\left(t, X(u, t), \mu_t,q\right)W(dq,dt), \\
X(u, 0)=u, \\
\mu_{t}=\mu_{0} \circ X(\cdot, t)^{-1} ,
\end{array}\right.
\end{equation*}
where the solution $X(u,t), \ u\in \mathbb R^d$ defines a stochastic flow $\{X(\cdot,t)\}_{t\geq0}$, and the measure $\mu_t$ evolves as the pushforward of the initial measure $\mu_0$ under this flow.
The measure $\mu_t$ can be treated as a mass distribution of the
interacting particle system.  This is the main difference with the McKean-Vlasov type equations.

For the stochastic maximum principle for optimal control problem, the adjoint method  is typically used to derive necessary conditions for optimality \cite{Kushner72,Bismut78,Bensoussan_1982,yong1999stochastic}, some related works refer to \cite{peng1990general,zhou1998stochastic,peng1999fully,han2010maximum,yong2013linear,han2013maximum}. More precisely, in the classical stochastic optimal control problem, where the system dynamics do not depend on the distribution of mass, the adjoint process is given by a solution $(Y, Z)$ to a backward stochastic differential equation (BSDE) \cite{pardoux1990adapted,peng1993backward,el1997backward},  which has the following form:
\begin{align*}
\left\{\begin{array}{ll}
-dY_t=[b_x(t,X_t,\alpha_t)Y_t+\sigma_x(t,X_t,\alpha_t)Z_t+f_x(t,X_t,\alpha_t)]dt-Z_tdW_t,
\\\quad Y_T=g_x(X_T).
\end{array}\right.
\end{align*}

In our case, the state equation (\ref{Control_problem_with_int}) depends on the evolving mass distribution $\mu_t$, as well as functions $f$ and $g$ in the cost functional (\ref{cost_function}). As a result, it becomes necessary to use derivatives with respect to probability measures to define the adjoint process for the optimal control problem with interaction.
In this paper, we construct the corresponding adjoint process and derive the necessary condition that an optimal control must satisfy. To this end, we introduce a generalized BSDE with interaction and prove the existence and uniqueness of its solution. We note that a related form of backward stochastic differential equation has been studied in \cite{DorJas2022backward}, but to apply it to a control problem with interaction, we must work with a more general formulation.
The main difficulty here lies in the fact that, in BSDEs with interaction, the coefficients depend on the random mass distribution generated by the flow. The well-posedness of such an equation is of independent mathematical interest.
By establishing the existence and uniqueness of solutions to a generalized BSDE, we derive an analogue of the Pontryagin stochastic maximum principle for optimal control for the equation with interaction.

The rest of this paper has the following structure. In Section \ref{sec_BS}, we provide a brief overview of the Brownian sheet. Although this material is well known, we include it for the reader’s convenience and to fix notation. In Section \ref{sec_Prem}, we recall the notion of Lions differentiability for functions defined on the space of probability measures.  We will use this type of differentiability to introduce an adjoint equation to a control problem with interaction. We define the generalized backward stochastic differential equation with interaction in Section \ref{sec_BSDE}. Here, we prove the existence and uniqueness of the solution to this equation. Section \ref{sec_CS} is devoted to the optimal control problem with interaction. The corresponding adjoint equation is given in the form of a backward stochastic differential equation with interaction from Section \ref{sec_BSDE}. We end this section with the maximum principle for optimal control. Finally, in Section \ref{sec_Example}, we present several examples to illustrate the main results.

\section{Brownian sheet}
\label{sec_BS}
In the differential equation with interaction, it is essential to use integrals with respect to the Brownian sheet in the stochastic term. This approach allows us to model a system of interacting particles with arbitrary correlation between any two of them.  For the reader’s convenience, we briefly recall the definition of the Brownian sheet and the associated stochastic calculus.
The notion of a Brownian sheet appears as a generalization of the Brownian motion to a multiparameter case  (see, for example, E. Wong and M. Zakai \cite{WongZakai1974}). 
 Let us  consider a Wiener process  $B$ on $\mathbb{R}^2$ with zero mean and covariance function 
 $$
 \mathbb{E} B(s_1,s_2) B(t_1,t_2) = \min(s_1, t_1) \min (s_2,t_2). 
 $$
Define a random Gaussian measure $W$ on the Borel $\sigma-$algebra $\mathcal B(\mathbb R^2)$ by
 $$
 W([a, b]\times [c,d]) = B(b,d) - B(a,d)-B(b,c)+B(a,c),
 $$
 for $a<b,$ $c<d.$
 Then $W([a,b]\times [c,d])\sim N(0, (b-a)(d-c))$ and for disjoint  rectangles
 $[a, b]\times [c,d],$ $[a', b']\times [c',d']$ random values
 $W([a,b]\times [c,d])$ and $W([a',b']\times [c',d'])$ are independent. This construction can be generalized for a multidimensional case.
 
Integration with respect to the Brownian sheet was defined and investigated by E. Wong and M. Zakai \cite{WongZakai1974} who also derived an It{\^o} formula for multiparameter processes defined via such integrals \cite{WONG1978339}.
 It is worth noting that the Brownian sheet arises as a limit process in the context of interacting particle systems, as shown in the works of K. Kuroda and H. Manaka \cite{KurodaManaka1986}, and K. Kuroda and H. Tanemura \cite{KurodaTanemura1988}. The idea of using a Wiener sheet to model stochastic systems driven by spatially correlated noise also appears in the study of random vortex models, such as in the work by P. Kotelenez \cite{Kotelenez92}.

A construction of the Brownian sheet using a sequence of independent Brownian motions can be found in the book by A. Dorogovtsev \cite{dorogovtsev2023measure}. 
Since we will use this representation in the present paper, we briefly describe it here for the reader's convenience. 
Let  $\lambda$ denote the Lebesgue measure in $\mathbb{R}^d$ and let $\left\{e_{k} ; k \geq 1\right\}$ be an orthonormal basis in $L_{2}(\mathbb{R}^d, \lambda)$. For a Borel set $\Delta \subset \mathbb{R}^d \times[0 ;+\infty)$ with finite Lebesgue measure, and for any $t\geq 0$ define time-section
$$
\Delta_{t}=\left\{r \in \mathbb{R}^{d}:(r, t) \in \Delta\right\}.
$$
By the Fubini theorem, for almost all $t,$ the indicator function
$
\mathbb{I}_{\Delta_{t}}$ belongs to $ L_{2}(\mathbb{R}^d, \lambda) .
$
Therefore, it admits an expansion in terms of the orthonormal basis:
$$
\mathbb{I}_{\Delta_{t}}=\sum_{k=1}^{\infty} f_{k}(t) e_{k}.
$$
Let $\{w_k\}_{k\geq 1}$ be a sequence of independent standard Wiener processes.  Now we define the Gaussian random variable associated with the set $\Delta$ as
$$
W(\Delta)=\sum_{k=1}^{\infty} \int_{0}^{+\infty} f_{k}(t) d w_{k}(t).
$$
It is easy to see that the following relations hold:\\
(1) $\mathbb{E} W(\Delta)=0,\  \mathbb{E} W(\Delta)^{2}=\lambda_{2}(\Delta)$.\\
(2) If $\Delta_{1}$ and $\Delta_{2}$ have finite Lebesgue measure, then
$$
\mathbb{E} W\left(\Delta_{1}\right) W\left(\Delta_{2}\right)=\lambda_{2}\left(\Delta_{1} \cap \Delta_{2}\right) .
$$
(3) If $\Delta_{1} \cap \Delta_{2}=\emptyset$, then
$$
W\left(\Delta_{1}\right)+W\left(\Delta_{2}\right)=W\left(\Delta_{1} \cup \Delta_{2}\right) .
$$

It is natural to associate with Brownian sheet $W$ the flow of $\sigma$-fields
$$
\mathcal{F}_{t}=\sigma\{W(\Delta): \Delta \subset \mathbb{R}^d \times[0 ; t]\}.
$$
It can be proved that for all $t \geq 0:$ 
$$
\mathcal{F}_{t}=\sigma\left\{w_{k}(s): k \geq 1, s \leq t\right\} .
$$
Now, consider a random function $f$ from $L_{2}\left(\mathbb{R}^d \times[0 ;+\infty), \lambda_{2}\right)$ such that,  for every $t \geq 0$  the restriction of $f$ to $\mathbb{R}^d \times[0 ; t]$ is $\mathcal{F}_{t}$-measurable.  To define the stochastic integral
$$
\int_{0}^{+\infty} \int_{\mathbb{R}^d} f(s, r) W(d s, d r) 
$$
we use a sequence of approximating step functions: 
$$
f_{n}(s, r)=\sum_{k=0}^{n} \varphi_{k}^{n}(r) \mathbb{I}_{\left[t_{k}^{n} ; t_{k+1}^{n}\right]}(s),
$$
with the following properties:\\
(1) $\mathbb{E} \int_{0}^{+\infty} \int_{\mathbb{R}^d}\left(f_{n}-f\right)^{2} d u d s \rightarrow 0, n \rightarrow \infty$,\\
(2) for every $k,$ $ \varphi_{k}^{n}$ is $\mathcal{F}_{t_{k}^{n}}$-measurable,\\
(3) for every $k$,
$$
\varphi_{k}^{n}(r)=\sum_{j=0}^{n-1} \alpha_{k j}^{n} \mathbb{I}_{\Delta_{j}^{n}} (r),
$$
where $\Delta_{j}^{n}, j=0, \ldots, n$ are disjoint subsets of $\mathbb{R}^d$ with finite measure.
For each $n \geq 1$, define the stochastic integral as
$$
\int_{0}^{+\infty} \int_{\mathbb{R}^d} f_{n}(s, r) W(d s, d r)=\sum_{k=0}^{n} \sum_{j=0}^{n} a_{k j}^{n} W\left(\left[t_{k}^{n} ; t_{k+1}^{n}\right) \times \Delta_{j}^{n}\right).
$$
Then, the stochastic integral of $f$ with respect to $W$ is defined as the limit in $L_2-$sense:
$$
\int_{0}^{+\infty} \int_{\mathbb{R}^d} f(s, r) W(d s, d r):=\lim _{n \rightarrow \infty} \int_{0}^{+\infty} \int_{\mathbb{R}^d} f_{n}(s, r) W(d s, d r).
$$
It can be checked that the value of the limit does not depend on the choice of the approximating sequence $\left\{f_{n} ; n \geq 1\right\}$ and that the obtained integral has properties:
$$
\begin{aligned}
& \mathbb{E} \int_{0}^{+\infty} \int_{\mathbb{R}^d} f(s, r) W(d s, d r)=0 \\
& \mathbb{E}\left(\int_{0}^{+\infty} \int_{\mathbb{R}^d} f(s, r) W(d s, d r)\right)^{2}=\mathbb{E} \int_{0}^{+\infty} \int_{\mathbb{R}^d} f(s, r)^{2} d s d u.
\end{aligned}
$$
Using the definition of $W,$ this integral can also be written as
\begin{align}\label{trans}
\int_{0}^{+\infty} \int_{\mathbb{R}^d} f(s, r) W(d s, d r) = \sum_{k=1}^{\infty}\int_0^{+\infty}g_k(s) dw_k(s),
\end{align}
where 
$g_k(t) = \int_{\mathbb R^d}f(t,r) e_k(r)dr.$
Using this representation, we obtain the following result for $\mathcal F_t-$adapted martingales.
\begin{lemma}\label{martin}
Let $M_t$ be a square integrable $\mathcal{F}_t$-adapted martingale. Then there exists $z(\cdot,\cdot)\in L^2([0,T]\times \mathbb{R}^d)$, such that 
\begin{align*}
M_t=\mathbb{E}M_0+\int_0^t\int_{\mathbb{R}^d} z(s,r)W(ds,dr).
\end{align*}
\end{lemma}
\begin{proof}
Let us denote by $[0,T]^n_< = \{0\leq t_1\leq, \ldots, \leq t_n\leq T\}.$
Using It$\hat{\rm o}$-Wiener expansion (Chapter 1 of \cite{dorogovtsev2019stochastic}), there exists $a_n\in L^2\left([0,T]^n_<\times\left(\mathbb{R}^d\right)^n\right)$ for $n=1,2,...$, such that
\begin{align*}
M_T=\mathbb{E}M_0+\sum_{n=1}^\infty\int_{[0,T]^n_<}\int_{(\mathbb{R}^d)^n}a_n(r_1,...,r_n,t_1,...,t_n)W(dt_1,dr_1)...W(dt_n,dr_n).
\end{align*}
Define
\begin{align*}
&z(t,r)\\&=\sum_{n=1}^\infty\int_{[0,T]^{n-1}_<}\int_{(\mathbb{R}^d)^{n-1}}a_n(r_1,...,r_{n-1},r,t_1,...,t_{n-1},t)W(dt_1,dr_1)...W(dt_{n-1},dr_{n-1})
\end{align*}
Then we can rewrite
\begin{align*}
M_T=\mathbb{E}M_0+\int_0^T\int_{\mathbb{R}^d}z(s,r)W(ds,dr).
\end{align*}
From this we get representation
\begin{align*}
M_t=\mathbb{E}^{\mathcal{F}_t}M_T=\mathbb{E}M_0+\int_0^t\int_{\mathbb{R}^d} z(s,r)W(ds,dr),
\end{align*}
where $\mathbb{E}^{\mathcal{F}_t}$ denotes a conditional expectation with respect to the $\sigma-$algebra $\mathcal{F}_t.$
\end{proof}

Using integrals with respect to the  Wiener sheet, we can construct two Brownian motions with an arbitrary correlation structure. Indeed, let the functions $f_1,$ $f_2$ be such that  $\int_{\mathbb R^d} f_i^2(r)dr = 1.$
Define two Brownian motions $B_i(t) = \int_0^t \int_{\mathbb R^d} f_i(r) W(ds,dr).$
Then the correlation between these Brownian motions is equal to 
$$
\mathbb E B_1(s) B_2(t) = t\wedge s \int_{\mathbb R^d}f_1(r) f_2(r) dr.
$$

Let $\{x(u,t), \ u\in \mathbb R, t\geq 0\}$ be a solution to the stochastic differential equation with interaction (\ref{SDE_WI}). Then the correlation  between the trajectories starting from initial positions $u$ and $v$ is given by
$$
\text{cor}(x(u,t), x(v,t)) = \int_0^t\int_{\mathbb{R}}
\mathbb E\sigma(s, x(u,s),\mu_s, r) \sigma(s, x(v,s),\mu_s, r)dsdr. 
$$
For example, if we choose the coefficient function as $$\sigma(s, u,\mu, r) = \exp\{-\|u-r\|^2\},$$
then we obtain
\begin{align*}
\text{cor}(x(u,t), x(v,t)) &= \mathbb E\int_0^t\int_{\mathbb{R}}
\mathbb \exp\{-\|x(u,s)-r\|^2\}\exp\{-\|x(v,s)-r\|^2\}dsdr\\
&=\pi^{d/2}\mathbb E\int_0^t
\mathbb \exp\{-\|x(u,s)-x(v,s)\|^2\}ds.
\end{align*}
This formula shows that the correlation between $ x(u,t)$ and $x(v,t)$  depends on the proximity of the two trajectories over time. Specifically, the exponential kernel $\exp\{-\|x(u,s)-x(v,s)\|^2\}$ decays rapidly when the paths diverge, meaning that points which stay close to each other during the evolution will exhibit stronger correlation. Such a structure naturally arises in models of spatially interacting particles, where local noise influences nearby components more strongly than distant ones.

\section{Differentiation of Function of Measure-Valued Argument}
\label{sec_Prem}
 As we mentioned in the Introduction, in order to define the adjoint process to the stochastic optimal control for the equation with interactions, we need to differentiate the coefficients of the equation with respect to a probability measure. 
There are several approaches to defining the differentiability of a real-valued function on a space of measures  (see, for example, \cite{dawson_measure_valued_processes,Villani,Mas2007} ). 
In this paper, we will use the concept of differentiability in the sense of P.-L. Lions for functions of measure (see, for example, \cite{cardaliaguet2010notes}). Let $L^p(\mathbb{R}^d,\mathcal{B},\mu; \mathbb{R}^d)$ denote the space of $\mathcal{B}$ measurable  variables $X:\mathbb{R}^d\to\mathbb{R}^d $ such that $$||X||_{L^p}=\left(\int_{\mathbb{R}^d}|x|^p\mu(dx)\right)^{\frac{1}{p}}<\infty.$$ Let $\mathcal{P}(\mathbb{R}^d)$ be the set of probability measures $\mu$ on  $(\mathbb{R}^d,\mathcal{B}(\mathbb{R}^d))$, and define
\begin{align*}
\mathcal{P}_2(\mathbb{R}^d):=\left\{\mu\in \mathcal{P}(\mathbb{R}^d): \int_{\mathbb{R}^d}|x|^2\mu(dx)< \infty\right\}.
\end{align*}
The 2-Wasserstein metric for $\mu^1, \mu^2\in\mathcal{P}_2(\mathbb{R}^d)$ is defined by
\begin{align*}
W_2(\mu^1, \mu^2):=\inf \Bigg\{\Big(&\int_{\mathbb{R}^d \times \mathbb{R}^d}|x-y|^2 \rho(d x, d y)\Big)^{\frac{1}{2}};\\&
\, \rho\in\mathcal{P}_2(\mathbb{R}^d\times\mathbb{R}^d),\,\rho(\cdot\times\mathbb{R}^d)=\mu^1(\cdot),\,\rho(\mathbb{R}^d\times\cdot)=\mu^2(\cdot)\Bigg\} .
\end{align*}

For any function $h: \mathcal{P}_2(\mathbb{R}^d)\to \mathbb{R}^d$, we define ``lifted" function $\Tilde{h}:L^2(\mathcal{B};\mathbb{R}^d)\to \mathbb{R}^d$  by $\tilde{h}(X)=h(\mu_X),$ where $ X\in L^2(\mathcal{B};\mathbb{R}^d)$ is a random variable with distribution $\mu_X.$
If for $\mu\in\mathcal{P}_2(\mathbb{R}^d)$, there exists $X\in L^2(\mathcal{B};\mathbb{R}^d)$ such that $\mu=\mu_X$ and $\tilde{h}:L^2(\mathcal{B};\mathbb{R}^d)\to \mathbb{R}^d$ is Fr\'echet differentiable at $X$, then $h: \mathcal{P}_2(\mathbb{R}^d)\to \mathbb{R}^d$ is said to be differentiable at $\mu$ and the derivative is defined as follows.
~\\

\begin{definition}
    Let $X\in L^2(\mathcal{B};\mathbb{R}^d)$. We say $\tilde{h}$ is Fr\'echet differentiable at $X$, if there exists a bounded linear operator $D\tilde{h}(X)\in L(L^2(\mathcal{B};\mathbb{R}^d),\mathbb{R}^d)$ such that for all $Y\in L^2(\mathcal{B};\mathbb{R}^d)$,
\begin{align}\label{2.1}
\tilde{h}(X+Y)-\tilde{h}(X)=D\tilde{h}(X)(Y)+o(||Y||_{L^2}),\quad as\quad ||Y||_{L^2}\to 0,
\end{align}
where $||Y||_{L^2}^2=\int_{\mathbb{R}^d}|y|^2\mu_Y(dy).$
\end{definition} 

Since $L^2(\mathcal{B};\mathbb{R}^d)$ is $Hilbert$ space, due to the Riesz representation theorem, and as shown by Lions \cite{lionscours}, there exists a Borel measurable function $g: \mathbb{R}^d\to \mathbb{R}^d$
such that
\begin{align*}
D\tilde{h}(X)(Y)=\int_{\mathbb{R}^d\times \mathbb{R}^d}g(x)\cdot y\rho(dx,dy),\quad P-a.s. \quad Y\in L^2(\mathcal{B};\mathbb{R}^d),
\end{align*}
where $\rho$ is the joint  distribution of $X$ and $Y$, and the function $g$ depends on $X$ only through its law $\mu_X$. Thus, we can write (\ref{2.1}) as
\begin{align*}
h(\mu_{X+Y})-h(\mu_X)=\int_{\mathbb{R}^d\times \mathbb{R}^d}g(x)\cdot y\rho(dx,dy)+o(||Y||_{L^2}).
\end{align*}
The function $g(\cdot)$ is called the Lions derivative of $h: \mathcal{P}_2(\mathbb{R}^d)\to \mathbb{R}$ at $\mu=\mu_X$ and it is denoted by $h_\mu(\mu,y)=g(y), \quad y\in\mathbb{R}^d$.

\begin{example}
    Consider the function
    $$a(u,\mu)=\int_{\mathbb{R}^d}K(u-v)\mu(dv), \ \mu\in\mathcal{P}_2(\mathbb{R}^d),$$ 
    where $K$ is a smooth function such that $$\sup_{u\in\mathbb{R}^d}|K'(u)|+|K''(u)|\le K_0.$$
    Then the lifted function $\tilde{a}(u,X)=\int_{\mathbb{R}^d}K(u-v)\mu(dv):=\overline{K(u-X)}$, where $X\sim\mu$. Let $Y$ be another random variable and $\|Y\|_{L^2}$ is small enough. Assume that the joint distribution of $X$ and $Y$ is $\rho(\cdot,\cdot)$. Using Taylor expansion $$K(u-x-y)=K(u-x)+K'(u-x)\cdot (-y)+\frac{1}{2}K''(u-x-\theta_xy)\cdot(-y)^2,\quad \theta_x\in[0,1],$$ we obtain
\begin{align*}
&\tilde{a}(u,X+Y)-\tilde{a}(u,X)\\&=\overline{K\left(u-(X+Y)\right)}-\overline{K(u-X)}\\&
=\int_{\mathbb{R}^d\times\mathbb{R}^d}\left(K(u-(x+y))-K(u-x)\right)\rho(dx,dy)\\&
=-\int_{\mathbb{R}^d\times\mathbb{R}^d}K'(u-x)y\rho(dx,dy)+\int_{\mathbb{R}^d\times\mathbb{R}^d}\frac{1}{2}K''(u-(x+\theta_xy))y^2\rho(dx,dy)
\\&=-\int_{\mathbb{R}^d\times\mathbb{R}^d}K'(u-x)y\rho(dx,dy)+o(\|Y\|_{L^2}),\quad \|Y\|_{L^2}\to 0.
\end{align*}
This shows that the Fr\'echet derivative is $D\tilde{a}(u,X)(Y)=-\int_{\mathbb{R}\times\mathbb{R}}K'(u-x)y\rho(dx,dy)$ and hence the Lions derivative of $a(u, \mu)$ is
$$a_\mu(u,\mu,x)=-K'(u-x).$$
\end{example}   

\section{Backward stochastic differential equation with interaction}
\label{sec_BSDE}
As we mentioned in the Introduction, to establish the stochastic maximum principle for optimal control of the equation with interaction, we will use the method of adjoint process. In this context, the adjoint processes satisfy a backward stochastic differential equation with interaction. This section presents the general BSDE with interaction and proves the existence and uniqueness of its solution. 
We consider the following generalized BSDE with interaction:
\begin{align}\label{3.1}
\left\{\begin{array}{ll}
-dy(u,t)&=f(u,t,y(u,t),z(u,t,\cdot),\mathcal{M}^{y,u}_t,\mathcal{M}^{z,u}_t)dt-\int_{\mathbb{R}^d}z(u,t,r)W(dt,dr),
\\\quad y(u,T)&=\xi(u), \quad u\sim \mu_0,
\\\qquad\mathcal{M}^{y,u}_t&=\mu_0\circ\Phi^{-1}(u,t,\cdot,y(\cdot,t)),
\\\qquad\mathcal{M}^{z,u}_t&=\mu_0\circ\int_{\mathbb{R}^d}\Psi^{-1}(u,t,\cdot,r,z(\cdot,t,r))dr,
\end{array}\right.
\end{align}
where $\mu_0\in \mathcal{P}_2(\mathbb{R}^d)$ is the given mass distribution, and $\xi:\Omega\times \mathbb{R}^d\to\mathbb{R}^d$ such that for all  $u, u_1, u_2, v\in\mathbb{R}^d, $ 
$$\mathbb{E}|\xi(u_1)-\xi(u_2)|^2\le L|u_1-u_2|^2.$$ 
The functions $\Phi(u,t,v,y)$ and $\Psi(u,t,v,r,z) $ are $\mathcal{F}_t$-adapted, defined on $\Omega\times \mathbb{R}^d\times [0,T]\times \mathbb{R}^d\times \mathbb{R}^d$ and  $\Omega\times \mathbb{R}^d\times [0,T]\times \mathbb{R}^d\times\mathbb{R}^d \times\mathbb{R}^{d\times m }$ with values in $\mathbb{R}^d$, respectively. The function $f(u,t,y,z(\cdot),\mu,\nu)$ is  $\mathcal{F}_t$-adapted on $\Omega\times \mathbb{R}^d\times[0,T]\times \mathbb{R}^d\times L^2(\mathbb{R}^d)\times \mathcal{P}_2(\mathbb{R}^d)\times \mathcal{P}_2(\mathbb{R}^d)$ with values in $\mathbb{R}^d$.

To formulate the definition of the solution to this generalized BSDE with interaction, we first introduce the following notations. Let $\mu_0$ be a given probability measure on $\mathbb R^d$.
\begin{itemize}
\item $\mathbb{H}_T^2\left(\mathbb{R}^d\right)$ denotes the space of all predictable processes $\varphi: \Omega \times\mathbb{R}^d\times[0, T] \mapsto \mathbb{R}^d$ such that $$\|\varphi\|^2:=\mathbb{E} \int_{\mathbb{R}^d}\int_0^T\left|\varphi(u,t)\right|^2 d t\mu_0(du)<+\infty.$$

\item For $\beta>0,$ 
define the exponential norm:
$$ \|\phi\|_\beta^2:=\mathbb{E} \int_{\mathbb{R}^d} \int_0^T e^{\beta t} \left|\phi(u,t)\right|^2 d t\mu_0(du) .$$

\item $ \mathbb{H}_{T,\beta}^2\left(\mathbb{R}^d\right)$ denotes the subspace  of $ \mathbb{H}_T^2\left(\mathbb{R}^d\right)$  endowed with the norm $\|\cdot\|_\beta.$

\item $ \mathbb{K}_{T,\beta}^2\left(\mathbb{R}^{d\times m}\right)$ denotes the space  of  measurable functions, $\psi: \Omega \times\mathbb R^d \times[0, T] \times \mathbb{R}^d\mapsto \mathbb{R}^{d\times m}$ such that (using the same notation for simplify) $$\|\psi\|_\beta^2=\mathbb{E}\int_{\mathbb{R}^d}\int_0^T\int_{\mathbb{R}^d}e^{\beta t}|\psi(u,t,r)|^2drdt\mu_0(du)<+\infty.$$

\end{itemize}
\begin{remark}
The choice of the exponential $\beta$-norm is inspired by \cite{el1997backward}. This norm is essential in constructing a contraction mapping in the proof of the following Lemma~\ref{unique_sol}.
\end{remark}

\begin{definition}
    A solution to the generalized BSDE with interaction (\ref{3.1}) is a pair of processes  $(y, z)\in \mathbb{H}_{T,\beta}^2\left(\mathbb{R}^d\right)\times\mathbb{K}_{T,\beta}^2\left(\mathbb{R}^{d\times m}\right)$ such that for all $t\in [0,T]$,
  \begin{align*}
     y(u,t) = &\xi(u) +\int_t^Tf(u,s,y(u,s),z(u,s,\cdot),\mathcal{M}^{y,u}_s,\mathcal{M}^{z,u}_s)ds\\&-\int_t^T\int_{\mathbb{R}^d}z(u,s,r)W(ds,dr),
  \end{align*}
    where 
    $$
   \mathcal{M}^{y,u}_t=\mu_0\circ\Phi^{-1}(u,t,\cdot,y(\cdot,t)),
\qquad\mathcal{M}^{z,u}_t=\mu_0\circ\int_{\mathbb{R}^d}\Psi^{-1}(u,t,\cdot,r,z(\cdot,t,r))dr.$$
\end{definition}

Notice that for fixed $(u,t)\in\mathbb{R}^d\times [0,T]$, the mappings $\Phi^{-1}(u,t,\cdot,y(\cdot,t))$ and
$$\int_{\mathbb{R}^d}\Psi^{-1}(u,t,\cdot,r,z(\cdot,t,r))dr$$ are from $\mathbb{R}^d$ to $\mathbb{R}^d$ and $\mathbb{R}^{d\times m}$ related to $y$ and $z$, respectively. It is clear that $\mathcal{M}^{y,u}_t$ and $\mathcal{M}^{z,u}_t$ are $(u,t)$-measurable measures. More precisely, for fixed $(u,t)\in\mathbb{R}^d\times[0,T]$, denoting by $\Delta_1=\left\{v\mid \Phi(u,t,v,y(v,t))\in\Delta\right\}$ we have $\mathcal{M}^{y,u}_t(\Delta)=\mu_0(\Delta_1),$ which is $\mathcal{F}_t$-measurable.

Let us now demonstrate how the solution of the following equation can be represented:
\begin{align}\label{yz}
y(u,t) = \xi(u) +\int_t^Tf(u,s)ds-\int_t^T\int_{\mathbb{R}^d}z(u,s,r)W(ds,dr).
\end{align}
In this case, the process $y$ is given by the conditional expectation:
\begin{align*}
y(u,t) = \mathbb{E}^{\mathcal{F}_t}\left[\xi(u) +\int_t^Tf(u,s)ds\right].
\end{align*}
Note that $\mathbb{E}^{\mathcal{F}_t}\left[\xi(u) +\int_0^Tf(u,s)ds\right]$ is a square integrable martingale and by Lemma \ref{martin}, we can get
\begin{align*}
\mathbb{E}^{\mathcal{F}_t}\left[\xi(u) +\int_0^Tf(u,s)ds\right]=y(u,0)+\int_0^t\int_{\mathbb{R}^d}z(u,s,r)W(ds,dr).
\end{align*}
Thus, we can write:
\begin{align*}
y(u,t)=y(u,0)-\int_0^tf(u,s)ds+\int_0^t\int_{\mathbb{R}^d}z(u,s,r)W(ds,dr),
\end{align*}
and it is straightforward to verify that $(y,z)$ satisfies equation (\ref{yz}).

To prove existence and uniqueness of of a solution to BSDE (\ref{3.1}), we define a mapping $I:(y,z)\mapsto (Y,Z)$ from $\mathbb{H}_{T,\beta}^2\left(\mathbb{R}^d\right)\times\mathbb{K}_{T,\beta}^2\left(\mathbb{R}^{d\times m}\right)$ onto itself by
\begin{align*}
 Y(u,t) = &\xi(u) +\int_t^Tf(u,s,y(u,s),z(u,s,\cdot),\mathcal{M}^{y,u}_s,\mathcal{M}^{z,u}_s)ds\\&-\int_t^T\int_{\mathbb{R}^d}Z(u,s,r)W(ds,dr).
\end{align*}
Our goal is to prove that $I$ is a contraction, which will then imply that BSDE (\ref{3.1}) has a unique solution. This is also the main idea in the proof of the existence and uniqueness of the solution in the next lemma.

We assume that $\mathbb{E}\int_0^T|f(u,s,y,z(\cdot),\mu,\nu)|^2ds<+\infty,$ and that there exist constants $L_1,\ L_2>0$ such that  for all $t\in[0,T],\ u,\ v_i,\ y_i\in\mathbb{R}^d$, $z_i\in L^2(\mathbb{R}^d)$, and $\mu_i, \nu_i\in\mathcal{P}_2(\mathbb{R}^d)$, $i=1,2,$ the following inequalities hold:
\begin{align}\label{3.2}
|f(u,t,y_1,z_1(\cdot),\mu_1,\nu_1)-f(u,t,y_2,z_2(\cdot),\mu_2,\nu_2)|^2\le L_1R_1^2,
\end{align}
and
\begin{align}\label{3.4}
&|\Phi(u,t,v_1,y_1)-\Phi(u,t,v_2,y_2)|^2\notag\\&+|\Psi(u,t,v_1,r,z_1(\cdot))-\Psi(u,t,v_2,r,z_2(\cdot))|^2\le L_2R_2^2,
\end{align}
where $$R_1^2=\left(|y_1-y_2|^2+\int_{\mathbb{R}^d}|z_1(r)-z_2(r)|^2dr+W_2^2(\mu_1,\mu_2)+W_2^2(\nu_1,\nu_2)\right),$$  and $$R_2^2=\left(|v_1-v_2|^2+|y_1-y_2|^2+\int_{\mathbb{R}^d}|z_1(r)-z_2(r)|^2dr\right).$$
 Here, $W_2(\cdot,\cdot)$ is the 2-Wasserstein metric

The following lemma ensures that, under the given assumptions, the mapping $I$ is a contraction, which will then imply that BSDE
(\ref{3.1}) has a unique solution. 
\begin{lemma} 
\label{unique_sol}
Assume that  the assumptions (\ref{3.2}) and  (\ref{3.4}) hold. Then the equation (\ref{3.1}) has a unique solution.
\end{lemma}
\begin{proof}
We use the method of constructing a contraction mapping, as described in Chapter 2 in \cite{el1997backward}. For any $\mathcal{F}_t$-adapted
 continuous process $y^1(u,t),$ $y^2(u,t),\ z^1(u,t,r)$, $z^2(u,t,r)$ with bounded $\beta$-norm, and measurable with respect to spatial variable, let
\begin{align*}
\left\{\begin{array}{ll}
-dY^i(u,t)=f(u,t,y^i(u,t),z^i(u,t,\cdot),\mathcal{M}^{y^i,u}_t,\mathcal{M}^{z^i,u}_t)dt\\\qquad\qquad\qquad-\int_{\mathbb{R}^d}Z^i(u,t,r)W(dt,dr),
\\\quad Y^i(u,T)=\xi(u),
\end{array}\right.
\end{align*}
for $i=1,2$. Our aim is to show that mapping $I:(y,z)\mapsto(Y,Z)$ is a contraction mapping under the $\beta$-norm.
Let us denote
\begin{align*}
\delta l_1(u,t)=l_1^1(u,t)-l_1^2(u,t),\\
\delta l_2(u,t,r)=l_2^1(u,t,r)-l_2^2(u,t,r)
\end{align*}
for $l_1=Y,y,\, l_2=Z,z$, and
\begin{align*}
\delta f(u,t)=&f(u,t,y^1(u,t),z^1(u,t,\cdot),\mathcal{M}^{y^1,u}_t,\mathcal{M}^{z^1,u}_t)\\&-f(u,t,y^2(u,t),z^2(u,t,\cdot),\mathcal{M}^{y^2,u}_t,\mathcal{M}^{z^2,u}_t).
\end{align*}
By It$\hat{\rm o}$'s formula, we have
\begin{align*}
d\left(e^{\beta t}\delta Y(u,t)^2\right)=&\beta e^{\beta t}\delta Y(u,t)^2dt+2e^{\beta t}\delta Y(u,t)d\delta Y(u,t)+e^{\beta t}\left(d\delta Y(u,t)\right)^2\notag\\
=&e^{\beta t}\left[\beta\delta Y(u,t)^2-2\delta Y(u,t)\delta f(u,t)+\int_{\mathbb{R}^d}\delta Z(u,t,r)^2dr\right]dt\\&+2e^{\beta t}\delta Y(u,t)\int_{\mathbb{R}^d}\delta Z(u,t,r)W(dt,dr).
\end{align*}
Taking the integral from $t$ to $T$ and taking the expectation, we obtain
\begin{align}\label{11}
&\mathbb{E}e^{\beta t}\delta Y(u,t)^2+\mathbb{E}\int_{\mathbb{R}^d}\int_t^Te^{\beta s}\delta Z(u,s)^2dsdr\\
=&\mathbb{E}\int_t^T e^{\beta s}\left[-\beta \delta Y(u,s)^2+2\delta Y(u,s)\delta f(u,s)\right]ds\notag\\
\le&\mathbb{E}\int_t^T e^{\beta s}\left[-\beta \delta Y(u,s)^2+\beta\delta Y(u,s)^2+\frac{1}{\beta}\delta f(u,s)^2\right]ds\notag\\
\le&\frac{L_1}{\beta}\mathbb{E}\int_t^T e^{\beta s}\Bigg[\delta y(u,s)^2+\int_{\mathbb{R}^d}\delta z(u,s,r)^2dr\notag\\&\qquad\qquad\qquad+W_2^2(\mathcal{M}_s^{y^1,u},\mathcal{M}_s^{y^2,u})+W_2^2(\mathcal{M}_s^{z^1,u},\mathcal{M}_s^{z^2,u})\Bigg]ds.\notag
\end{align}
Using the Lipschitz condition (\ref{3.4}), we estimate the Wasserstein distances as 
\begin{align*}
W_2^2(\mathcal{M}_s^{y^1,u},\mathcal{M}_s^{y^2,u})&\le \int_{\mathbb{R}}\Big|\Phi\big(u,s,v,y^1(v,s)\big)-\Phi\big(u,s,v,y^2(v,s)\big)\Big|^2\mu_0(dv)\notag\\
&\le L_2\int_{\mathbb{R}^d}\delta y(v,s)^2\mu_0(dv).
\end{align*}
So by Fubini's theorem, we have that
\begin{align}\label{14}
\mathbb{E}\int_0^T e^{\beta s} W_2^2(\mathcal{M}_s^{y^1,u},\mathcal{M}_s^{y^2,u})ds\le L_2\|\delta y\|^2_{\beta}.
\end{align}
In the same way, we also get
\begin{align}\label{15}
\mathbb{E}\int_0^T e^{\beta s} W_2^2(\mathcal{M}_s^{z^1,u},\mathcal{M}_s^{z^2,u})ds\le L_2\|\delta z\|^2_{\beta}.
\end{align}
Using equality (\ref{11}) we have
\begin{align}\label{16}
&\quad \mathbb{E}e^{\beta t}\delta Y(u,t)^2\\&\le\frac{L_1}{\beta}\mathbb{E}\int_t^T e^{\beta s}\Bigg[\delta y(u,s)^2+\int_{\mathbb{R}^d}\delta z(u,s,r)^2dr\notag\\&\qquad\qquad\qquad\qquad+W_2^2(\mathcal{M}_s^{y^1,u},\mathcal{M}_s^{y^2,u})+W_2^2(\mathcal{M}_s^{z^1,u},\mathcal{M}_s^{z^2,u})\Bigg]ds\notag\\
&\le\frac{L_1}{\beta}\mathbb{E}\int_0^T e^{\beta s}\Bigg[\delta y(u,s)^2+\int_{\mathbb{R}^d}\delta z(u,s,r)^2dr\notag\\&\qquad\qquad\qquad\qquad+W_2^2(\mathcal{M}_s^{y^1,u},\mathcal{M}_s^{y^2,u})+W_2^2(\mathcal{M}_s^{z^1,u},\mathcal{M}_s^{z^2,u})\Bigg]ds\notag,
\end{align}
and
\begin{align}\label{16b}
&\quad \mathbb{E}\int_t^Te^{\beta s}\int_{\mathbb{R}^d}\delta Z(u,s,r)^2drds\\
&\le\frac{L_1}{\beta}\mathbb{E}\int_t^T e^{\beta s}\Bigg[\delta y(u,s)^2+\int_{\mathbb{R}^d}\delta z(u,s,r)^2dr
\notag\\&\qquad\qquad\qquad\qquad+W_2^2(\mathcal{M}_s^{y^1,u},\mathcal{M}_s^{y^2,u})+W_2^2(\mathcal{M}_s^{z^1,u},\mathcal{M}_s^{z^2,u})\Bigg]ds.\notag
\end{align}
Integrating the inequality (\ref{16}) over $u$ with respect to $\mu_0$ and over $t\in[0,T]$, and using (\ref{14}), (\ref{15}), we obtain
\begin{align*}
\|\delta Y\|^2_{\beta}\le \frac{L_1(1+L_2)T}{\beta}\Big(\|\delta y\|^2_{\beta}+\|\delta z\|^2_{\beta}\Big).
\end{align*}
Let $t=0$ and integrating (\ref{16b}) over $u$ with respect to $\mu_0$, we get
\begin{align*}
\|\delta Z\|^2_{\beta}\le \frac{L_1(1+L_2)}{\beta}\Big(\|\delta y\|^2_{\beta}+\|\delta z\|^2_{\beta}\Big).
\end{align*}
Combining the two last inequalities, we have
\begin{align*}
\|\delta Y\|^2_{\beta}+\|\delta Z\|^2_{\beta}\le \frac{L_1(1+L_2)(T+1)}{\beta}\Big(\|\delta y\|^2_{\beta}+\|\delta z\|^2_{\beta}\Big).
\end{align*}

Choosing $\beta >L_1(1+L_2)(T+1)$, we see that this mapping $I:(y,z)\mapsto(Y,Z)$ is a 
contraction mapping from $\mathbb{H}_{T,\beta}^2\left(\mathbb{R}^d\right)\times\mathbb{K}_{T,\beta}^2\left(\mathbb{R}^{d\times m}\right)$ 
onto itself. Therefore, there exists a unique fixed point, which is the unique continuous solution of the BSDE (\ref{3.1}).
\end{proof}

In the next lemma, we show that if the terminal process $\xi$ is Lipschitz continuous in $L^{2k},$ then the solution $y(\cdot, t)$  to the generalized BSDE with interaction inherits this property for any $t\in [0, T].$ 
\begin{lemma}  Assume that $\xi$ satisfy   $\mathbb{E}|\xi(u)-\xi(v)|^{2k}\le M_k|u-v|^{2k}, k\ge1$ for every $u,v\in\mathbb{R}^d,$ and (\ref{3.2}), (\ref{3.4}) hold. Then there exist a constant $C_k>0,$ depending on $M_k, L_1, L_2$ and $T$, such that $$\mathbb{E}|y(u,t)-y(v,t)|^{2k}\le C_k|u-v|^{2k},\quad \forall t\in [0,T].$$
\end{lemma}
\begin{proof} We first consider the case $k=1.$  Denote
$$\Delta l_t=l(u,t)-l(v,t)$$
for $l=y,z$, and
$$\Delta f_t=f(u,t,y(u,t),z(u,t,\cdot),\mathcal{M}^{y,u}_t,\mathcal{M}^{z,u}_t)-f(v,t,y(v,t),z(v,t,\cdot),\mathcal{M}^{y,v}_t,\mathcal{M}^{z,v}_t).$$
Then
\begin{align*}
d|\Delta y_t|^2=&2\Delta y_td\Delta y_t+(d\Delta y_t)^2\\
=&-2\Delta y_t\Delta f_tdt+\int_{\mathbb{R}^d}|\Delta z_t(r)|^2drdt+2\Delta y_t\int_{\mathbb{R}^d}\Delta z_t(r)W(dt,dr).
\end{align*}
Integrating from $t$ to $T$ and taking expectations, we obtain:
\begin{align*}
\mathbb{E}|\Delta y_t|^2=&\mathbb{E}|\xi(u)-\xi(v)|^2+2\mathbb{E}\int_t^T \Delta y_s\Delta f_sds-\mathbb{E}\int_t^T \int_{\mathbb{R}^d}|\Delta z_t(r)|^2drds\\
\le &M_1|u-v|^2+\mathbb{E}\int_t^T \left(c|\Delta y_s|^2+\frac{1}{c}|\Delta f_s|^2\right)ds\\
&-\mathbb{E}\int_t^T \int_{\mathbb{R}^d}|\Delta z_t(r)|^2drds.
\end{align*}
Using the condition (\ref{3.4}), we can get the upper bound for the Wasserstein distance:
\begin{align*}
&W_2^2(\mathcal{M}_t^{y,u},\mathcal{M}_t^{y,v})+W_2^2(\mathcal{M}_t^{z,u},\mathcal{M}_t^{z,v})\\
\le& \int_{\mathbb{R}^d}\Big|\Phi\big(u,t,m,y(m,t)\big)-\Phi\big(v,t,m,y(m,t)\big)\Big|^2\mu_0(dm)\\
&+\int_{\mathbb{R}^d}\int_{\mathbb{R}^d}|\Psi\big(u,t,m,r,z(m,t,r)\big)-\Psi\big(v,t,m,r,z(m,t,r)\big)\Big|^2dr\mu_0(dm)\\
\le&L_2\int_{\mathbb{R}^d}|u-v|^2\mu_0(dm)
=L_2|u-v|^2.
\end{align*}
Therefore,
\begin{align}\label{38}
|\Delta f_t|^2\le& L_1\Bigg(|u-v|^2+|\Delta y_t|^2+\int_{\mathbb{R}^d}|\Delta z_t(r)|^2dr\\&\qquad+W_2^2(\mathcal{M}_t^{y,u},\mathcal{M}_t^{y,v})+W_2^2(\mathcal{M}_t^{z,u},\mathcal{M}_t^{z,v})\Bigg)\notag\\
\le&L_1(1+L_2)|u-v|^2+L_1\Big(|\Delta y_t|^2+\int_{\mathbb{R}^d}|\Delta z_t(r)|^2dr\Big).
\end{align}
Taking $c=2L_1$, we obtain
\begin{align*}
&\mathbb{E}|\Delta y_t|^2+\frac{1}{2}\int_t^T\int_{\mathbb{R}^d}|\Delta z_t(r)|^2drds\\\le& \left(M_1+\frac{(1+L_2)T}{2}\right)|u-v|^2+\frac{1}{2}\int_t^T |\Delta y_s|^2ds.
\end{align*}
By Gronwall's inequality, we deduce the result for $k=1$. Moreover, we also get 
\begin{align}\label{39}
\mathbb{E}\int_0^T \int_{\mathbb{R}^d}|z(u,s,r)-z(v,s,r)|^2drds\le \tilde{C} |u-v|^2
\end{align}
for some $\tilde{C}>0$.
For general $k\geq1$, a similar argument yields
\begin{align*}
\mathbb{E}|\Delta y_t|^{2k}=&\mathbb{E}|\xi(u)-\xi(v)|^{2k}+2k\mathbb{E}\int_t^T \Delta y_s^{2k-1}\Delta f_sds\\&-k(2k-1)\mathbb{E}\int_t^T |\Delta y_s|^{2k-2}\int_{\mathbb{R}^d}|\Delta z_s(r)|^2drds.
\end{align*}
Notice that $\Delta y_s^{2k-1}\Delta f_s\le\frac{1}{2}(\Delta |y_s|^{2k}+|\Delta y_s|^{2k-2}|\Delta f_s|^2)$. In the same way, the conclusion is proved.
\end{proof}

\section{Stochastic Maximum Principle}
\label{sec_CS}
In this section, we consider the following control problem. The state equation is given by
\begin{align*}
\left\{\begin{array}{ll}
dX(u,t) =b\Big(t,X(u,t) ,\mu_t,\alpha_t \Big)dt+\int_{\mathbb{R}^d}\sigma\Big(t,X(u,t) ,\mu_t,\alpha_t \Big)W(dt,dr),
\\X(u,0)=u,
\\\mu_t=\mu_0\circ X(\cdot,t)^{-1},\qquad 0\le t\le T,
\end{array}\right.
\end{align*}
where $\mu_t=\mu_0\circ X(\cdot,t)^{-1}$
represents the mass distribution of all particles
at time $t.$ The cost functional is defined as
\begin{align}\label{ccost}
    J(\alpha_t)=\mathbb{E}\left[\int_0^T  f\Big(t,\mu_t,\alpha_t \Big)dt+g(\mu_T)\right],
\end{align}
where $b(t,x,\mu,\alpha)$ and $\sigma(t,x,\mu,\alpha,r) $ are measurable functions on $\mathbb{R}\times \mathbb{R}^d\times \mathcal{P}_2(\mathbb{R}^d)\times \mathbb{R}^k $ and  $\mathbb{R}\times \mathbb{R}^d\times \mathcal{P}_2(\mathbb{R}^d)\times \mathbb{R}^k\times \mathbb{R}^d,$ respectively, taking values in $\mathbb{R}^d$ and $\mathbb{R}^{d\times m}.$ Both are assumed to be differentiable with respect to $x, \mu, \alpha.$
The functions $f(t,\mu,\alpha)$ and $g(\mu) $ are  measurable functions on $\mathbb{R}\times \mathcal{P}_2(\mathbb{R}^d)\times \mathbb{R}^k$ and $\mathcal{P}_2(\mathbb{R}^d)$, respectively, with values in $\mathbb{R}$, and are differentiable with respect to $ \mu$ and $\alpha.$

We assume that
\begin{align*}
|b(t,x,\mu,\alpha)|+\left(\int_{\mathbb{R}^d}|\sigma(t,x,\mu,\alpha,r)|^2dr\right)^{\frac{1}{2}
}\le L(1+|x|+W_2(\mu,\delta_0)),
\end{align*}
and
\begin{align*}
&|b(t,x_1,\mu_1,\alpha_1)-b(t,x_2,\mu_2,\alpha_2)|+\\&\left(\int_{\mathbb{R}^d}|\sigma(t,x_1,\mu_1,\alpha_1,r)-\sigma(t,x_2,\mu_2,\alpha_2,r)|^2dr\right)^\frac{1}{2}\le LR_3,
\end{align*}
for some constant $L>0$, where $R_3=\left(|x_1-x_2|+W_2(\mu_1,\mu_2)+|\alpha_1-\alpha_2|\right)$ and $\delta_0$ is the Dirac measure at $0$. We also assume that $b_x(t,x,\mu,\alpha),$ $ b_\alpha(t,x,\mu,\alpha),$  $b_\mu(t,x,\mu,\alpha)(v),$ $\sigma_x(t,x,\mu,\alpha,r),$ $ \sigma_\alpha(t,x,\mu,\alpha,r),$ $\sigma_\mu(t,x,\mu,\alpha,r)(v)$ are continuous w.r.t. $(x,\mu,\alpha,v)$.

We denote by $\mathbb{U}$ the set of admissible controls
$\alpha=(\alpha_t)_{0\le t\le T}$, which take values in a given closed-convex set $\textbf{U}\subset \mathbb{R}^k$ and satisfying $\mathbb{E}\int_0^T |\alpha_t|^2 dt <\infty$.

To simplify the notation without loss of generality, we consider the case $d=m=k=1$. Assume $\alpha_t^{*}$ is an optimal control process, i.e.,
\begin{align*}
    J(\alpha_t^*)=\min_{\alpha_t\in \mathbb{U}}J(\alpha_t)
.\end{align*}
For all $ 0<\varepsilon<1$, let
\begin{align*}
    \alpha_t^\varepsilon =(1-\varepsilon)\alpha_t^*+\varepsilon\tilde{\alpha}_t:=
\alpha^*_t+\varepsilon \beta_t,
\end{align*}
where $\tilde{\alpha}_t$ is any other admissible control.
Let $X^\varepsilon(u,t),$ and $ X^*(u,t)$ be the state process  corresponding to $\alpha^\varepsilon_t,$ and $ \alpha^*_t$, respectively.

Define $V(u,t)$ as a solution of
\begin{align}\label{4.333}
\left\{\begin{array}{ll}
dV(u,t)=\Big[b_x^*(u,t)V(u,t)+\int_{\mathbb{R}}b_\mu^*(u,t)(v)V(v,t)\mu_0(dv)+b_\alpha^*(u,t)\beta_t\Big]dt\\
\qquad\qquad\quad+\Big[\sigma_x^*(u,t,r)V(u,t)+\\\qquad\qquad\qquad\quad\int_{\mathbb{R}}\int_{\mathbb{R}}\sigma_\mu^*(u,t,r)(v)V(v,t)\mu_0(dv)+\sigma_\alpha^*(u,t)\beta_t\Big]W(dt,dr),
\\V(u,0)=0,
\end{array}\right.
\end{align}
where
\begin{align*}
b_x^*(u,t)=b_x\left(t,X^*(u,t),\mu_t^*,\alpha^*_t\right),&\quad
b_\alpha^*(u,t)=b_\alpha\left(t,X^*(u,t),\mu_t^*,\alpha^*_t\right),\\
\sigma_x^*(u,t,r)=\sigma_x\left(t,X^*(u,t),\mu_t^*,\alpha^*_t,r\right),&\quad
\sigma_\alpha^*(u,t,r)=\sigma_\alpha\left(t,X^*(u,t),\mu_t^*,\alpha^*_t,r\right),\\
b^*_\mu(u,t)(v)=b_\mu&\Big(t,X^*(u,t),\mu_t^*,\alpha^*_t\Big) (X^*(v,t)), \\\sigma^*_\mu(u,t,r)(v)=\sigma_\mu&\Big(t,X^*(u,t),\mu_t^*,\alpha^*_t,r\Big)(X^*(v,t)).
\end{align*}

\begin{lemma}\label{lemma5.2}
Let $V(u,t)$ be  defined by the equation (\ref{4.333}). Then we have
\begin{align*}
\sup_{0\le t\le T}\lim_{\varepsilon\to 0} \mathbb{E}\int_\mathbb{R}\left[\frac{X^\varepsilon(u,t)-X^*(u,t)}{\varepsilon}-V(u,t)\right]^2\mu_0(du)=0.
\end{align*}
\end{lemma}
\begin{proof} 
 Let $\Delta X(u,t)=X^\varepsilon(u,t)-X^*(u,t)$ and $$\tilde{X}(u,t)=\frac{X^\varepsilon(u,t)-X^*(u,t)}{\varepsilon}-V(u,t).$$ Then,
\begin{align*}
\Delta X(u,t)=\int_0^tb^\varepsilon(u,s)-b^*(u,s)ds+\int_{\mathbb{R}}\int_0^t\sigma^\varepsilon(u,s,r)-\sigma^*(u,s,r) W(ds,dr),
\end{align*}
where
\begin{align*}
b^\varepsilon(u,t)=b(t,X^\varepsilon(u,t),\mu^\varepsilon_t,\alpha^\varepsilon_t),\quad  \sigma^\varepsilon(u,t,r)=\sigma(t,X^\varepsilon(u,t),\mu^\varepsilon_t,\alpha^\varepsilon_t,r).
\end{align*}
By the Lipschitz condition,  we have
\begin{align}\label{DX}
\mathbb{E}|\Delta X(u,t)|^2\le& C\mathbb{E}\int_0^t|b^\varepsilon(u,s)-b^*(u,s)|^2ds\\&+C\mathbb{E}\int_\mathbb{R}\int_0^t|\sigma^\varepsilon(u,s,r)-\sigma^*(u,s,r)|^2dsdr\notag\\
\le&C\mathbb{E}\int_0^t \left(|\Delta X(u,t)|^2+\int_{\mathbb{R}}|\Delta X(v,s)|^2\mu_0(dv)+\varepsilon^2\beta_s^2\right)ds,
\end{align}
where $C$ is a constant independent of $u,t,$ and $\varepsilon$, which may vary from line to line. Integrating both sides with respect to $\mu_0(du),$ we obtain
\begin{align*}
\mathbb{E\int_\mathbb{R}}|\Delta X(u,t)|^2\mu_0(du)\le C\mathbb{E}\int_0^t\int_\mathbb{R}|\Delta X(u,s)|^2\mu_0(du)ds+\varepsilon^2 C\|\beta\|^2.
\end{align*}
By applying Gronwall's inequality, we conclude that
\begin{align}\label{DeltaX}
\mathbb{E\int_\mathbb{R}}|\Delta X(u,t)|^2\mu_0(du)=O(\varepsilon^2),\quad \varepsilon\to 0, \quad \forall t\in[0,T].
\end{align}
Combining the inequalities (\ref{DX}) and (\ref{DeltaX}) and applying Gronwall's inequality again, we have
\begin{align}\label{DDX}
\mathbb{E}|\Delta X(u,t)|^2=O(\varepsilon^2), \quad \varepsilon\to 0, \quad \forall u\in \mathbb{R},\quad t\in[0,T].
\end{align}

Now consider $\tilde{X}(u,t).$ We have 
\begin{align*}
&\tilde{X}(u,t)\\=&\int_0^t\left[b_x^*(u,s)\tilde{X}(u,s)+\int_{\mathbb{R}}b_\mu^*(u,s)(v)\tilde{X}(v,s)\mu_0(dv)\right]ds\\
&+\int_{\mathbb{R}}\int_0^t\Big[\sigma_x^*(u,s,r)\tilde{X}(u,s)+\int_{\mathbb{R}}\sigma_\mu^*(u,s,r)(v)\tilde{X}(v,s)\mu_0(dv)\Big]W(ds,dr)\\&+R^\varepsilon(u,t),
\end{align*}
where the remaining term $R^{\varepsilon}$ is given by
\begin{align*}
&R^\varepsilon(u,t)\\=&\int_0^t\left[\tilde{b}_x^\varepsilon(u,s)-b_x^*(u,s)\right]\varepsilon^{-1}\Delta X(u,s) ds\\&+\int_0^t\int_\mathbb{R}\left[\tilde{b}_\mu^\varepsilon(u,s)(v)-b_\mu^*(u,s)(v)\right]\varepsilon^{-1}\Delta X(v,s)\mu_0(dv) ds\\
&+\int_{\mathbb{R}}\int_0^t\left[\tilde{\sigma}_x^\varepsilon(u,s,r)-\sigma_x^*(u,s,r)\right]\varepsilon^{-1}\Delta X(u,s)W(ds,dr)\\
&+\int_\mathbb{R}\int_0^t\int_\mathbb{R}\left[\tilde{\sigma}_\mu^\varepsilon(u,s,r)(v)-\sigma_\mu^*(u,s,r)(v)\right]\varepsilon^{-1}\Delta X(v,s)\mu_0(dv)W(ds,dr)\\
&+\int_0^t\left[\tilde{b}_\alpha^\varepsilon(u,s)-b_\alpha^*(u,s)\right]\beta_s ds+\int_{\mathbb{R}}\int_0^t\left[\tilde{\sigma}_\alpha^\varepsilon(u,s,r)-\sigma_\alpha^*(u,s,r)\right]\beta_sW(ds,dr)
\end{align*}
with
\begin{align*}
&\tilde{b}^\varepsilon(u,t)\\&=\int_0^1b\left(t,X^*(u,t)+\lambda\Delta X(u,t),\mu_0\circ\left[X^*(\cdot,t)+\lambda\Delta X(\cdot,t)\right]^{-1},\alpha^*_t+\lambda\varepsilon\beta_t\right)d\lambda,\\
&\tilde{\sigma}^\varepsilon(u,t,r)\\&=\int_0^1\sigma\left(t,X^*(u,t)+\lambda\Delta X(u,t),\mu_0\circ\left[X^*(\cdot,t)+\lambda\Delta X(\cdot,t)\right]^{-1},\alpha^*_t+\lambda\varepsilon\beta_t,r\right)d\lambda.
\end{align*}
By the continuous assumption on $b_x,b_\alpha,b_\mu,\sigma_x,\sigma_\alpha,\sigma_\mu$ and using (\ref{DDX}), it follows that $R^\varepsilon(u,t)\to0$ as $\varepsilon\to 0$.
Therefore, 
\begin{align*}
\mathbb{E}\tilde{X}(u,t)^2\le C\mathbb{E}\int_0^t\left[\tilde{X}(u,s)^2+\left(\int_{\mathbb{R}}\tilde{X}(v,s)\mu_0(dv)\right)^2\right]ds+o(1).
\end{align*}
Integrating with respect to $\mu_0$, and applying Fubini's theorem and Gronwall's inequality, we conclude that
\begin{align*}
\mathbb{E}\int_\mathbb{R}|\tilde{X}(u,t)|^2\mu_0(du)=o(1),\quad \varepsilon\to 0,\quad\forall t\in[0,T],
\end{align*}
which completes the proof of Lemma \ref{lemma5.2}
\end{proof} 

Similar to the proof of above lemma, we have that
\begin{align}\label{3.7}
\frac{J(\alpha_t^\varepsilon)-J(\alpha_t^*)}{\varepsilon}\to
\mathbb{E}\Bigg[\int_0^T\left(\int_{\mathbb{R}}f_\mu^*(t)(v)V(v,t)\mu_0(dv)+f_\alpha^*(t)\beta_t\right)dt\notag\\+\int_{\mathbb{R}}g_\mu^*(v)V(v,T)\mu_0(dv)\Bigg]
,\end{align}
as $\varepsilon\to 0$, where
\begin{align*}
    g_\mu^*(v)=g_\mu(\mu_T^*)(X^*(v,T)),&\qquad
f_\alpha^*(t)=f_\alpha\left(t,\mu_t^*,\alpha^*_t\right),\\
f^*_\mu(t)(v)=f_\mu\Big(t,&\mu_t^*,\alpha^*_t\Big) (X^*(v,t)).
\end{align*}

Now we consider the Hamiltonian system and the stochastic maximum principle. Define the Hamiltonian function $H$ by
\begin{align}
    H(t,x,\mu,\alpha,p,q(\cdot))=b(t,x,\mu,\alpha)p+\int_{\mathbb{R}}\sigma (t,x,\mu,\alpha,r)q(r)dr+f(t,\mu,\alpha).
\end{align}
We then define the adjoint equation as
\begin{align}\label{absde}
\left\{\begin{array}{ll}
-dp(u,t)=&\Big[b_x^*(u,t)p(u,t)+\int_{\mathbb{R}}b_\mu^*(v,t)(u)p(v,t)\mu_0(dv)\\\\&\quad+\int_{\mathbb{R}}\sigma_x^*(u,t,r)q(u,t,r)dr\\
\\
&+\int_{\mathbb{R}}\int_{\mathbb{R}}\sigma_\mu^*(v,t,r)(u)q(v,t,r)\mu_0(dv)dr+f_\mu^*(t)(u)\Big]dt\\\\&-q(u,t,r)W(dt,dr),
 \\
\\\quad p(u,T)=&g_\mu^*(u),
\end{array}\right.
\end{align}
which can be written compactly as
\begin{align*}
\left\{\begin{array}{ll}
-dp(u,t)=\Big[H^*_x(u,t)+\int_{\mathbb{R}}H_\mu^*(v,t)(u)\mu_0(dv)\Big]dt -q(u,t,r)W(dt,dr),
 \\
\\p(u,T)=g_\mu^*(u),
\end{array}\right.
\end{align*}
where
\begin{align*}
    H^*(u,t)=H\Big(t,X^*(u,t) ,\mu^*_t,\alpha^*_t,p(u,t),q(u,t,\cdot)\Big).
\end{align*}
In the following lemma, we derive the expression for the Gâteaux derivative of the cost functional.
\begin{lemma}
\label{derivative}
 Let $(\alpha_t^*)_{0\le t\le T}$ be the optimal control process,  $X^*$ be the corresponding state process, and $(p,q)$ be the adjoint process satisfying (\ref{absde}). Then, the $G\hat{a}teaux$ derivative of $J$ at $\alpha^*_t$ in the direction $\left(\tilde{\alpha}_t-\alpha_t^*\right)$ is
\begin{align*}
    \frac{d}{d\varepsilon}J\left(\alpha_t^*+\varepsilon \left(\tilde{\alpha}_t-\alpha_t^*\right)\right)\Big|_{\varepsilon=0}=\mathbb{E}\int_{\mathbb{R}}\int_0^T [H^*_\alpha(u,t)\cdot \left(\tilde{\alpha}_t-\alpha_t^*\right)]dt\mu_0(du)
,\end{align*}
where $\tilde{\alpha}_t$ is any other control process. 
\end{lemma}
\begin{proof} 
Further, we will use the notation $\beta_t=\tilde{\alpha}_t-\alpha_t^*$ for simplicity.
By It$\hat{\rm o}$'s formula, we have
\begin{align}\label{3.10}
 &\quad d(p(u,t)V(u,t))\\&=p(u,t)dV(u,t)+V(u,t)dp(u,t)+dp(u,t)dV(u,t) \notag\\
 &=p(u,t)\Big[b_x^*(u,t)V(u,t)+\int_{\mathbb{R}}b_\mu^*(u,t)(v)V(v,t)\mu_0(dv)+b_\alpha^*(u,t)\beta_t\Big]dt\notag \\
 &\quad -V(u,t)\Big[b_x^*(u,t)p(u,t)+\int_{\mathbb{R}}b_\mu^*(v,t)(u)p(v,t)\mu_0(dv)\notag\\&\quad+\int_{\mathbb{R}}\sigma_x^*(u,t,r)q(u,t,r)dr\notag\\
&\quad+\int_{\mathbb{R}}\sigma_\mu^*(v,t)(u)q(v,t)\mu_0(dv)+f_\mu^*(t)(u)\Big]dt\notag \\
&\quad+V(u,t)\int_{\mathbb{R}}\sigma_x^*(u,t,r)q(u,t,r)drdt+\int_{\mathbb{R}}\sigma_\alpha^*(u,t,r)q(u,t,r)\beta_tdrdt\notag\\
&\quad+\int_{\mathbb{R}}q(u,t,r)\int_{\mathbb{R}}\sigma_\mu^*(u,t,r)(v)V(v,t)\mu_0(dv)drdt+\int_{\mathbb{R}}M(u,t,r)W(dt,dr) \notag \\
&=\Big[p(u,t)\int_{\mathbb{R}}b_\mu^*(u,t)(v)V(v,t)\mu_0(dv)-V(u,t)\int_{\mathbb{R}}b_\mu^*(v,t)(u)p(v,t)\mu_0(dv)\Big]dt
    \notag\\
&\quad+\int_{\mathbb{R}}\Big[q(u,t,r)\int_{\mathbb{R}}\sigma_\mu^*(u,t,r)(v)V(v,t)\mu_0(dv)\notag\\&\qquad\quad-V(u,t)\int_{\mathbb{R}}\sigma_\mu^*(v,t,r)(u)q(v,t,r)\mu_0(dv)\Big]drdt
    \notag\\
    &\quad-V(u,t)f_\mu^*(t)(u)dt+\Big[b_\alpha^*(u,t)p(u,t)+\int_{\mathbb{R}}\sigma_\alpha^*(u,t,r)q(u,t,r)dr\Big]\beta_tdt\notag\\&\quad+M(u,t,r)W(dt,dr)
,\notag
\end{align}
where $M(u,\cdot,r)$ is an $\mathcal{F}_t$-adapted process.

By Fubini's theorem, we have
\begin{align}\label{3.15}
&\quad\int_{\mathbb{R}}p(u,t)\int_{\mathbb{R}}b_\mu^*(u,t)(v)V(v,t)\mu_0(dv)\mu_0(du)\\&=\int_{\mathbb{R}}V(v,t)\int_{\mathbb{R}}b_\mu^*(u,t)(v)p(u,t)\mu_0(du)\mu_0(dv)\notag\\
&=\int_{\mathbb{R}}V(u,t)\int_{\mathbb{R}}b_\mu^*(v,t)(u)p(v,t)\mu_0(dv)\mu_0(du).\notag
\end{align}
In the same way, we get
\begin{align}\label{3.11}
&\quad\int_{\mathbb{R}}\int_{\mathbb{R}}q(u,t,r)\int_{\mathbb{R}}\sigma_\mu^*(u,t,r)(v)V(v,t)\mu_0(dv)\mu_0(du)dr
\\&=\int_{\mathbb{R}}V(u,t)\int_{\mathbb{R}}\int_{\mathbb{R}}\sigma_\mu^*(v,t,r)(u)q(v,t,r)dr\mu_0(dv)\mu_0(du).\notag
\end{align}
Taking  expectation and integrating with respect to  $t$ and  $\mu_0$ in equality (\ref{3.10}), and using  (\ref{3.15}) and (\ref{3.11}), we obtain
\begin{align}\label{3.14}
  &\mathbb{E}\int_{\mathbb{R}}g_\mu^*(u)V(u,T)\mu_0(du)\\&=\mathbb{E}\int_{\mathbb{R}}p(u,T)V(u,T)\mu_0(du)=\mathbb{E}\int_{\mathbb{R}}\int_0^T d(p(u,t)V(u,t))dt\mu_0(du)\notag\\
    &=-\mathbb{E}\int_{\mathbb{R}}\int_0^T\Big[V(u,t)f_\mu^*(t)(u)\Big]dt\mu_0(du)\notag\\
    &\quad+\mathbb{E}\int_{\mathbb{R}}\int_0^T\Big[b_\alpha^*(u,t)p(u,t)+\int_{\mathbb{R}}\sigma_\alpha^*(u,t,r)q(u,t,r)dr\Big]\beta_tdt\mu_0(du).\notag
\end{align}
Substituting (\ref{3.14}) into (\ref{3.7}),  we get
\begin{align}
    &\quad \frac{d}{d\varepsilon}J(\alpha_t^*+\varepsilon \beta_t)\Big|_{\varepsilon=0}\\&=\mathbb{E}\int_{\mathbb{R}}\int_0^T\Big[\Big(b_\alpha^*(u,t)p(u,t)+\int_{\mathbb{R}}\sigma_\alpha^*(u,t,r)q(u,t,r)dr+f_\alpha^*(t)\Big)\cdot \beta_t\Big]dt\mu_0(du)\notag\\
&=\mathbb{E}\int_{\mathbb{R}}\int_0^T[H_\alpha^*(u,t)\cdot \beta_t]dt\mu_0(du),\notag
\end{align}
which completes the proof.
\end{proof}

The following result provides a necessary condition for optimality in terms of the Hamiltonian function.
\begin{theorem} Assume that  $(\alpha_t^*)_{0\le t\le T}$ is the optimal control process, $(X_t^*)_{0\le t\le T}$ and $(p_t,q_t)_{0\le t\le T}$ are corresponding state process and adjoint process satisfying (\ref{absde}), respectively, $\mu^*_t=\mu_0\circ X^*(\cdot,t)^{-1}$ . Then
\begin{align}\label{3.25}
    \left[\int_{\mathbb{R}}H_\alpha^*(u,t)\mu_0(du)\right]\cdot(\tilde{\alpha}_t-\alpha^*_t)\ge 0 ,
    \qquad \forall \tilde{\alpha}_t \in \mathbb{U},\quad d\lambda\otimes dP\quad a.s.
\end{align}
\end{theorem}
~\\

\begin{proof}
Since $(\alpha_t^*)_{0\le t\le T}$ is the optimal control process, we have  
\begin{align*}
    \frac{d}{d\varepsilon}J\Big(\alpha_t^*+\varepsilon (\tilde{\alpha}_t-\alpha_t^*)\Big)\Big|_{\varepsilon=0}\ge 0
.\end{align*}
By Lemma \ref{derivative}, we get
\begin{align*}
\mathbb{E}\int_{\mathbb{R}}\int_0^T\Big[H_\alpha^*(u,t)\cdot (\tilde{\alpha}_t-\alpha_t^*)\Big]dt\mu_0(du)\ge 0
.\end{align*}
Hence, for any   $t\in[0,T]$ and any $\mathcal{F}t$-measurable set $\mathcal{A}$
\begin{align*}
    \mathbb{E}\left[\mathbf{1}_{\mathcal{A}}\mathbb\int_{\mathbb{R}}H_\alpha^*(u,t)\cdot (\tilde{\alpha}_t-\alpha_t^*)\right]\mu_0(du)\ge 0.
\end{align*}
This implies that
\begin{align*}
   \left[\int_{\mathbb{R}}H_\alpha^*(u,t)\mu_0(du)\right]\cdot(\tilde{\alpha}_t-\alpha^*_t)\ge 0 ,
    \qquad \forall \tilde{\alpha}_t \in \mathbb{U},\quad d\lambda\otimes dP\quad a.s.
\end{align*}
\end{proof}

\begin{remark}
If the optimal control process $(\alpha_t^*)_{0\le t\le T}$ takes values in the interior of the $\textbf{U}$ , then condition (\ref{3.25}) can be replaced by 
\begin{align*}
    \int_{\mathbb{R}}H_\alpha^*(u,t)\mu_0(du)=0
.\end{align*}
~\\

Thus, we give the Hamiltonian system, which is a forward-backward SDE with interaction.
 \begin{align}\label{4.16}
    \left\{\begin{array}{ll}
dX^*(u,t)=H_p^*(u,t)dt+\int_{\mathbb{R}}\sigma^*(u,t,r)W(dt,dr),\\\\
  -dp_t=\Big[H^*_x(u,t)+\int_{\mathbb{R}}H_\mu^*(v,t)(u)\mu_0(dv)\Big]dt -q(u,t,r)W(dt,dr),\\\\
  X^*(u,0)=u,\\\\
 p(u,T)=g_\mu^*(u),\\\\
  \mu_t^*=\mu_0\circ X^*(\cdot,t)^{-1}.
\end{array}\right.
\end{align}
And the optimal control process should satisfy:
\begin{align}\label{4.17}
   \int_{\mathbb{R}}H_\alpha^*(u,t)\mu_0(du)=0,\quad\forall t\in[0,T],
\end{align}
where
\begin{align*}
     H^*(u,t)=&H\Big(t,X^*(u,t) ,\mu^*_t,\alpha^*_t,p(u,t),q(u,t,\cdot)\Big),\\
H(t,x,\mu,\alpha,p,q(\cdot))&=b(t,x,\mu,\alpha)p+\int_{\mathbb{R}}\sigma (t,x,\mu,\alpha,r)q(r)dr+f(t,\mu,\alpha).
\end{align*}
\end{remark}

\section{Linear Quadratic Case}
\label{sec_Example}

In this section, we consider a one-dimensional linear-quadratic (LQ) problem. Our goal is to control a system of particles so that their mass distribution is close to a given measure while minimizing the process cost at the same time. The system is given by
\begin{align}\label{55}
\left\{\begin{array}{ll}
dX(u,t) =\Big(AX(u,t)+B\Bar{X}_t+C\alpha_t\Big)dt\\\quad\qquad\qquad+\int_\mathbb{R}\varphi(t,r)\Big(DX(u,t)+F\Bar{X}_t+H\alpha_t\Big)W(dt,dr),
\\X(u,0)=u,
\\\mu_t=\mu_0\circ X(\cdot,t)^{-1},\qquad 0\le t\le T,
\end{array}\right.
\end{align}
where $\Bar{X}_t=\int_{\mathbb{R}}u\mu_t(du)=\int_{\mathbb{R}}X(u,t)\mu_0(du)$ and $\varphi$ is a given function such that $\int_\mathbb{R}\varphi(t,r)^2dr<+\infty$ for all $t\in[0,T]$.
The cost functional is given by
\begin{align}\label{56}
J(\alpha_t)=\frac{1}{2}\mathbb{E}\int_0^T\left(Q\int_{\mathbb{R}}x^2\mu_t(dx)+S\Bar{X}^2_t+R\alpha_t^2\right)dt+\mathbb{E}\rho^2(\mu_T,\nu),
\end{align}
where constants $Q,S\ge 0, R>0$, and the distance $\rho$ between measures is given by
\begin{align*}
\rho^2(\mu,\nu)=\int\int_{\mathbb{R}^2}(u-v)^2[\mu(du)-\nu(du)]\cdot[\nu(dv)-\mu(dv)],
\end{align*}
with $\nu$  a given measure.

\begin{lemma} Let $h(\mu)=\rho^2(\mu,\nu)$. Then the differential of $h$ at $\mu_T^*$ is given by
\begin{align}
h_\mu(\mu^*_T)(X^*(u,T))=4\left(\bar{X}^*_T-\int_{\mathbb{R}}v\nu(dv)\right),\qquad \forall u\in\mathbb{R},
\end{align}
where
\begin{align*}
\bar{X}^*_T=\int_{\mathbb{R}}x\mu^*_T(dx)=\int_{\mathbb{R}}X^*(u,T)\mu_0(du).
\end{align*}
\end{lemma}
~\\

\begin{proof} 
By the definition of function $h,$ we can write
\begin{align*}
h(\mu)&=\int\int_{\mathbb{R}^2}(u-v)^2\left[-\mu(du)\mu(dv)+\mu(du)\nu(dv)+\mu(dv)\nu(du)-\nu(du)\nu(dv)\right]\\
&=-\int\int_{\mathbb{R}^2}(u-v)^2\mu(du)\mu(dv)+2\int\int_{\mathbb{R}^2}(u-v)^2\mu(du)\nu(dv)\\
&\quad-\int\int_{\mathbb{R}^2}(u-v)^2\nu(du)\nu(dv)\\
&= -h_1(\mu)+2h_2(\mu)-h_3(\nu).
\end{align*}
Let $\delta X(u,T)=X^\varepsilon(u,T)-X^*(u,T)$. To compute the G\^ateux derivative, we calculate the difference $h_1(\mu^\varepsilon_T)-h_1(\mu_T^*):$
\begin{align*}
&\quad h_1(\mu^\varepsilon_T)-h_1(\mu_T^*)\\&=\int\int_{\mathbb{R}^2}(u-v)^2\mu^\varepsilon_T(du)\mu^\varepsilon_T(dv)-\int\int_{\mathbb{R}^2}(u-v)^2\mu^*_T(du)\mu^*_T(dv)\\
&=\int\int_{\mathbb{R}^2}\left[\Big(X^\varepsilon(u,T)-X^\varepsilon(v,T)\Big)^2-\Big(X^*(u,T)-X^*(v,T)\Big)^2\right]\mu_0(du)\mu_0(dv)\\
&=2\int\int_{\mathbb{R}^2}\left[\Big(X^*(u,T)-X^*(v,T)\Big)\Big(\delta X(u,T)-\delta X(v,T)\Big)\right]\mu_0(du)\mu_0(dv)+o(\varepsilon)\\
&=2\int\int_{\mathbb{R}^2}\Big[X^*(u,T)\delta X(u,T)-X^*(u,T)\delta X(v,T)-X^*(v,T)\delta X(u,T)\\
&\quad\qquad\qquad+X^*(v,T)\delta X(v,T)\Big]\mu_0(du)\mu_0(dv)+o(\varepsilon)\\
&=4\int_{\mathbb{R}} X^*(u,T)\delta X(u,T)\mu_0(du)-4\int\int_{\mathbb{R}^2} X^*(v,T)\delta X(u,T)\mu_0(du)\mu_0(dv)+o(\varepsilon).
\end{align*}
For $h_2(\mu^\varepsilon_T)-h_2(\mu_T^*)$ we have:
\begin{align*}
&\quad h_2(\mu^\varepsilon_T)-h_2(\mu_T^*)\\&=\int\int_{\mathbb{R}^2}(u-v)^2\mu^\varepsilon_T(du)\nu(dv)-\int\int_{\mathbb{R}^2}(u-v)^2\mu^*_T(du)\nu(dv)\\
&=\int\int_{\mathbb{R}^2}\left[\Big(X^\varepsilon(u,T)-v\Big)^2-\Big(X^*(u,T)-v\Big)^2\right]\mu_0(du)\nu(dv)
\end{align*}
\begin{align*}
&=2\int\int_{\mathbb{R}^2}\Big[\Big(X^*(u,T)-v\Big)\delta X(u,T)\Big]\mu_0(du)\nu(dv)+o(\varepsilon)\\&=2\int_{\mathbb{R}}X^*(u,T)\delta X(u,T)\mu_0(du)-2\int\int_{\mathbb{R}^2}v\delta X(u,T)\mu_0(du)\nu(dv)+o(\varepsilon).
\end{align*}
 Collecting terms, we have
\begin{align*}
h(\mu^\varepsilon_T)-h(\mu^*_T)&=-\left(h_1(\mu^\varepsilon_T)-h_1(\mu_T^*)\right)+2\left(h_2(\mu^\varepsilon_T)-h_2(\mu_T^*)\right)\\
&=-4\int\int_{\mathbb{R}^2}v\delta X(u,T)\mu_0(du)\nu(dv)\\&\quad+4\int\int_{\mathbb{R}^2} X^*(v,T)\delta X(u,T)\mu_0(du)\mu_0(dv)+o(\varepsilon).
\end{align*}
From this, the G\^ateux derivative of $h$ is given by
\begin{align*}
h_\mu(\mu_T^*)(X^*(u,T))=4\int_{\mathbb{R}}X^*(v,T)\mu_0(dv)-4\int_{\mathbb{R}}v\nu(dv):=4\left(\bar{X}^*_T-\int_{\mathbb{R}}v\nu(dv)\right).
\end{align*}
\end{proof}
\begin{remark}  For any other probability measure $\mu_1, \mu_2$ and the corresponding random variable $X_1, X_2$, same as the proof above, we can get that 
\begin{align}\label{58}
h(\mu_2)-h(\mu_1)\ge \ \big<h_\mu(\mu_1)(X_1),X_2-X_1\big>,
\end{align}
where $\big<X,Y\big>=\int_{\mathbb{R}^2}xy\rho(dx,dy)$. Furthermore,
\begin{align}\label{59}
h(\mu_1)+h(\mu_2)\ge2h(\mu_3),
\end{align}
where $\mu_3$ is the distribution of $\frac{X_1+X_2}{2}$. The inequality (\ref{58}), (\ref{59}) will be useful in the following.
\end{remark}

From (\ref{4.16}) and (\ref{4.17}), the adjoint equation is given by
\begin{align}\label{510}
\left\{\begin{array}{ll}
-dp(u,t)=\Big[Ap(u,t)+B\Bar{p}_t+D\int_\mathbb{R}\varphi(r)q(u,t,r)dr\\\\\qquad\qquad\qquad+F\int_\mathbb{R}\varphi(t,r)\Bar{q}_t(r)dr+QX^*(u,t)+S\Bar{X}_t^*\Big]dt\\\\\qquad\qquad\qquad\quad-q(u,t,r)W(dt,dr),
 \\
\\p(u,T)=4\left(\bar{X}^*_T-\int_{\mathbb{R}}v\nu(dv)\right),
\end{array}\right.
\end{align}
where $\Bar{p}_t=\int_{\mathbb{R}}p(u,t)\mu_0(du)$ and $\Bar{q}_t(r)=\int_{\mathbb{R}}q(u,t,r)\mu_0(du)$.
The necessary condition for optimal control is
\begin{align}\label{5.8}
\alpha_t^*=-R^{-1}\left(C\Bar{p}_t+H\int_\mathbb{R}\varphi(t,r)\Bar{q}_t(r)dr\right),\quad t\in[0,T].
\end{align}

\begin{theorem}  The function $\alpha_t^*=-R^{-1}\left(C\Bar{p}_t+H\int_\mathbb{R}\varphi(t,r)\Bar{q}_t(r)dr\right), t\in[0,T],$ is the unique optimal control for LQ problem (\ref{55}), (\ref{56}), where $(p_t,q_t)$ satisfy equation (\ref{510}). 
\end{theorem}
\begin{proof}  Let any $\alpha_t\subset \mathbb{U}$ be any admissible control, and let $X_t$ and $X_t^*$ be the state processes corresponding to $\alpha_t$ and $\alpha_t^*$, respectively. Denote $\delta X(u,t)=X(u,t)-X^*(u,t)$ and $\delta \alpha_t=\alpha_t-\alpha^*_t.$ Applying It$\hat{\rm o}$'s formula to $p(u,t)\delta X(u,t)$ we obtain
\begin{align*}
&dp(u,t)\delta X(u,t)\\=&B\left(p(u,t)\delta \Bar{X}_t-\Bar{p}_t\delta X(u,t)\right)+F\int_\mathbb{R}\varphi(t,r)\left(q(u,t,r)\delta \Bar{X}_t-\Bar{q}_t(r)\delta X(u,t)\right)drdt\\
&-\delta X(u,t)\left(QX^*(u,t)+S\Bar{X}^*_t\right)dt+\delta\alpha_t\left(Cp(u,t)+Hq(u,t)dt\right)\\&+\int_\mathbb{R}M^u(t,r)W(dt,dr).
\end{align*}
Using (\ref{58}) and formula (\ref{5.8}), we derive
\begin{align*}
&\mathbb{E}[h(\mu_T)-h(\mu^*_T)]\\
\ge& \mathbb{E}\int_{\mathbb{R}}\int_0^Tp(u,t)\delta X(u,t)dt\mu_0(du)\\
=&-\mathbb{E}\int_0^T\left[Q\int_{\mathbb{R}}X^*(u,t)\delta X(u,t)\mu_0(du)+S\Bar{X}^*_t\delta \Bar{X}_t+R\alpha_t^*\delta\alpha_t\right]dt.
\end{align*}
Applying the inequality  $a^2-b^2\ge 2b(a-b)$, we get
\begin{align*}
&J(\alpha_t)-J(\alpha^*_t)\\&=\frac{1}{2}\mathbb{E}\left[Q\int_{\mathbf
{R}}\left(X(u,t)^2-X^*(u,t)^2\right)\mu_0(du)+S(\Bar{X}^2_t-\Bar{X}_t^{*2})+R(\alpha_t^2-{\alpha_t^*}^2)\right]dt\\&\quad+\mathbb{E}[h(\mu_T)-h(\mu^*_T)]\\
&\ge\frac{1}{2}\mathbb{E}\bigg[Q\int_{\mathbf
{R}}\left(X(u,t)^2-X^*(u,t)^2\right)\mu_0(du)-2Q\int_{\mathbb{R}}X^*(u,t)\delta X(u,t)\mu_0(du)\\
&\quad +S(\Bar{X}^2_t-\Bar{X}_t^{*2})-2S\Bar{X}^*_t\delta \Bar{X}_t+R(\alpha_t^2-{\alpha_t^*}^2)-2R\alpha_t^*\delta\alpha_t\bigg]dt
\ge0.
\end{align*}
This shows that $\alpha_t^*$ is an optimal control. 

Now let us prove that $\alpha^*_t$ is unique. Assume that both $\alpha_t^{*,1}$ and $\alpha_t^{*,2}$ are optimal controls, $X^1(u,t)$ and $X^2(u,t)$ are corresponding state processes, respectively. It is easy to get $\frac{X^1(u,t)+X^2(u,t)}{2}$ is the corresponding state process to $\frac{\alpha_t^{*,1}+\alpha_t^{*,2}}{2}$. We assume there exists a constant $ \theta\ge 0$, such that
\begin{align*}
J(\alpha_t^{*,1})=J(\alpha_t^{*,2})=\theta.
\end{align*}
Using the identity $a^2+b^2=2[(\frac{a+b}{2})^2+(\frac{a-b}{2})^2]$ and by (\ref{59}), we have that
\begin{align*}
2\theta=&J(\alpha_t^{*,1})+J(\alpha_t^{*,2})\\
=&\frac{1}{2}\mathbb{E}\int_0^T\bigg[Q\int_{\mathbb{R}}\Big(X^1(u,t)X^1(u,t)+X^2(u,t)X^2(u,t)\Big)\mu_0(du)\\
&\qquad\qquad+S\Big(\Bar{X}^1_t\Bar{X}^1_t+\Bar{X}^2_t\Bar{X}^2_t\Big)+R\Big(\alpha_t^{*,1}\alpha_t^{*,1}+\alpha_t^{*,2}\alpha_t^{*,2}\Big)\bigg]dt\\&\qquad+\mathbb{E}\Big(\rho^2(\mu_T^{*,1},\nu)+\rho^2(\mu_T^{*,2.},\nu)\Big)\\
\ge&\mathbb{E}\int_0^T\bigg[Q\int_{\mathbb{R}}\Big(\frac{X^1(u,t)+X^2(u,t)}{2}\Big)^2\mu_0(du)+S\Big(\frac{\Bar{X}^1_t+\Bar{X}^2_t}{2}\Big)^2\\
&\qquad\qquad+R\Big(\frac{\alpha_t^{*,1}+\alpha_t^{*,2}}{2}\Big)^2+R\Big(\frac{\alpha_t^{*,1}-\alpha_t^{*,2}}{2}\Big)^2\bigg]dt+2\mathbb{E}\rho^2\Big(\mu_T^{*,3},\nu\Big)\\
=&2J\Big(\frac{\alpha_t^{*,1}+\alpha_t^{*,2}}{2}\Big)+\mathbb{E}\int_0^TR\Big(\frac{\alpha_t^{*,1}-\alpha_t^{*,2}}{2}\Big)^2dt\\
\ge& 2\theta+\frac{R}{4}\mathbb{E}\int_0^T|\alpha_t^{*,1}-\alpha_t^{*,2}|^2dt,
\end{align*}
where $\mu_T^{*,3}$ is the distribution of $\frac{X_T^{*,1}+X_T^{*,2}}{2}$.
Thus, we have that
\begin{align*}
    \mathbb{E}\int_0^T|\alpha_t^{*,1}-\alpha_t^{*,2}|^2dt\le 0,
\end{align*}
which shows that $\alpha_t^{*,1}=\alpha_t^{*,2}$.
\end{proof}
\begin{remark}
Consider another terminal condition of the cost functional (\ref{56}), let $$\tilde{\rho}^2(\mu_T)=\int\int_{\mathbb{R}^2}(u-v)^2\mu_T(du)\mu_T(dv).$$ By the similar method with Lemma 5.1, we can get the derivative of $\tilde{\rho}^2$ at $\mu^*_T$ is $$\partial _\mu\tilde{\rho}^2(\mu^*_T)(X^*(u,T))=4\left(X^*(u,T)-\bar{X}_T^*\right)$$ and 
the same conclusion as (\ref{58}), (\ref{59}). So that we can also get the unique optimal control in the same way by changing the terminal value of adjoint equation (\ref{510}) to $\partial _\mu\tilde{\rho}^2(\mu^*_T)(X^*(u,T))$.
\end{remark}

\bibliography{main}
\bibliographystyle{unsrt}

\end{document}